\documentclass[12pt]{amsart}
\usepackage{amsmath,amssymb,amsbsy,amsfonts,%
latexsym,amsopn,amstext,%
amsxtra,euscript,amscd}
\usepackage{url}
\usepackage{color}
\usepackage{float}
\usepackage{comment}
\begin{document}

\newtheorem{theorem}{Theorem}
\newtheorem{lemma}[theorem]{Lemma}
\newtheorem{claim}[theorem]{Claim}
\newtheorem{cor}[theorem]{Corollary}
\newtheorem{prop}[theorem]{Proposition}
\newtheorem{question}[theorem]{Question}
\newtheorem{remark}[theorem]{Remark}
\newtheorem{setting}[theorem]{Setting}
\newcommand{\hh}{{{\mathrm h}}}
\newcommand{\N}{\mathbb{N}}
\newcommand{\Q}{\mathbb{Q}}
\newcommand{\F}{\mathbb{F}}
\newcommand{\Z}{\mathbb{Z}}
\newcommand{\R}{\mathbb{R}}
\newcommand{\C}{\mathbb{C}}
\newcommand{\eps}{\varepsilon}
\newcommand{\NN}{\mathcal{N}}
%
\newcommand{\uq}{\mathbf{q}}
\newcommand{\ur}{\mathbf{r}}
\newcommand{\uh}{\mathbf{h}}
\newcommand{\ux}{\mathbf{x}}
\newcommand{\uj}{\mathbf{j}}
\newcommand{\md}{\textup{~mod~}}
\newcommand{\ph}{\varphi}

\numberwithin{equation}{section}
\numberwithin{theorem}{section}
\numberwithin{table}{section}

\title[Large sieve estimate for multivariate polynomial moduli]%
{Large sieve estimate for multivariate polynomial moduli and applications}
\date{\today}

\author{Karin Halupczok \and
  Marc Munsch}
  
\newcommand{\Addresses}{{
  \bigskip
  \footnotesize

  \medskip

  \textsc{Karin Halupczok, Mathematisches Institut der
 Heinrich-Heine-Universit\"{a}t D\"{u}sseldorf,
Universit\"{a}tsstra\ss{}e 1,
D-40225 D\"{u}sseldorf, Germany}\par\nopagebreak
  \textit{E-mail address:} \texttt{karin.halupczok@uni-duesseldorf.de}

   \textsc{Marc Munsch,  Institut f\"{u}r Analysis und Zahlentheorie 
8010 Graz, Steyrergasse 30, Graz, Austria}\par\nopagebreak
  \textit{E-mail address:} \texttt{munsch@math.tugraz.at}

}}


\begin{abstract}
  We prove large sieve inequalities with multivariate polynomial moduli and deduce a general Bombieri--Vinogradov type theorem for a class of polynomial moduli having a sufficient number of variables compared to its degree. This sharpens previous results of the first author in two aspects: the range of the moduli as well as
  the class of polynomials which can be handled. 
  As a consequence, we deduce that there exist infinitely many primes $p$
  such that $p-1$ has a prime divisor of size $\gg p^{2/5+o(1)}$
  that is the value of an incomplete norm form polynomial.
\end{abstract}

\bibliographystyle{plain}
\maketitle

\section{Introduction}

The large sieve in its arithmetic form was introduced by Linnik in the early 40s to bound the number of exceptions to Vinogradov's conjecture on the size of the least quadratic non-residue. Since then it has been the object of intensive study by mathematicians such as R\'{e}nyi, Roth, Bombieri and others leading to the modern presentation in its analytic form called large sieve inequality. The classic form of the large sieve inequality and its generalizations is an extremely powerful tool and has many applications in analytic number theoretical problems. For instance, as a consequence of a high dimensional version of the large sieve, Gallagher
\cite{Gall}
proved the  ‘generic’ irreducibility and maximality of the Galois group for integral polynomials with bounded height. A more recent focus involves the large sieve over sparse sequences of moduli, such as powers or more general polynomials (see also \cite{MS} for results on arbitrary sparse sequences). Obtaining large sieve inequalities in this context has various applications in diverse arithmetic problems. Without being exhaustive, let us mention few examples such as the distribution of primes in sparse 
 progressions~\cite{BZsparse,BakerPP1},  the existence of shifted primes divisible by a large square~\cite{Matomaki,Mer},  elliptic curves~\cite{BPS,SZell} or the study of Fermat quotients~\cite{FermatBFKS,IgorFermat}. \\ 
 
 The case of moduli defined as values of polynomials of several variables received much less attention. However, recently the first author obtained a general result in \cite{Karinquart} using  multidimensional Weyl sum estimates and highlighted several possible applications to problems related to the multiplicative structure of consecutive integers.\\  
 
 Recently, the second author \cite{Munsch} improved in some range of the parameters the existing large sieve inequalities with polynomial moduli in one variable. A crucial ingredient was the use of bounds on the number of solutions to polynomial equations in boxes. In Section \ref{multsieve} we explain how this idea can be adapted in our multivariable setting to give large sieve bounds which are sharper than those of \cite{Karinquart} in the case of a general multivariate polynomial. Additionally, these multidimensional estimates allow us to deduce interesting consequences. As an application, we obtain in Section \ref{BVsection} a Bombieri--Vinogradov type theorem for a general class of polynomial moduli having a sufficiently large number of variables compared to the degree. The proof involves a very careful splitting of the range of the parameters where both classical large sieve and our new results are needed. In Section \ref{adm}, we discuss an interesting choice of polynomials that can be made and in particular, the existence of infinitely many primes $p$ such that $p-1$ is divisible by certain divisors of multidimensional polynomial shape.

 \section{Multidimensional polynomial large sieve}\label{multsieve}
 \subsection{Notation and conventions}
Throughout the paper, the notation $U = O(V)$, 
$U \ll V$ and $ V\gg U$  are equivalent to $|U|\leqslant c V$ for some positive constant $c$, 
which depends on the degree of the polynomials involved and, where applies, on the coefficients of the polynomials. 

For any quantity $V> 1$ we write $U = V^{o(1)}$ (as $V \to \infty$) to indicate a function of $V$ which 
satisfies $ V^{-\varepsilon} \le |U| \le V^\varepsilon$ for any $\varepsilon> 0$, provided $V$ is large enough. One additional advantage 
of using $V^{o(1)}$ is that it absorbs $\log V$ and other similar quantities without changing  the whole 
expression.  \\

\subsection{Setting of the problem}

  In this section we consider a polynomial $P \in
\mathbb{Z}[X_1,\dots,X_\ell]$ in $\ell$ variables of total degree $k \geq
2$. For a real number $Q\geq 1$, we consider $\ell$- tuples
$\mathbf{q}=(q_1,\dots,q_\ell) \sim Q$ where $\mathbf{q} \sim Q$ means
that $\mathbf{q}$ is in the dyadic $Q$- box
$\displaystyle{\prod_{i=1}^{\ell}[Q,2Q)}$. Let $\{a_n\}_{n\geq 1}$ denote a
sequence of complex numbers and $M,N$ positive integers.  We define
 
 \[\sum_{N,Q,P}:=\sum_{\mathbf{q} \sim Q} \sum_{\substack{1\leq a \leq
   P(\mathbf{q}) \\ (a,P(\mathbf{q}))=1}} 
  \left\vert S\left(\frac{a}{P(\mathbf{q}))}\right)\right\vert^2\] 
where as usual
  \begin{equation}
\label{eq:eqq}
S(\theta):= \sum_{M<n\leq M+N} a_n \mathrm{e}(n\theta)
\end{equation}  and  $\mathrm{e}(z)=\exp(2i\pi z)$ for $z \in \mathbb{C}$. Our main goal is to obtain large sieve inequalities which are bounds of the shape
\begin{equation}\label{largesieve} \sum_{N,Q,P} \ll 
  \Delta_{k,\ell}(Q,N) \sum_{n=M+1}^{M+N} \vert a_n\vert^2,
 \end{equation} where $\Delta_{k,\ell}(Q,N)$ is some function of the parameters $N$ and $Q$ (which could both depend on $k$ and $\ell$) and the implied constant may also depend on the parameters $k$ and $\ell$.   The reader may notice that repetitions are allowed in the definition of $\sum_{N,Q,P}$ meaning that it is possible to have $P(\mathbf{q})=P(\mathbf{q'})$ for different $\ell$- tuples $\mathbf{q}$ and $\mathbf{q'}$ \footnote{We could have stated our results removing repetitions or weighting every moduli by the inverse of the number of appearances. For our purpose these formulations are essentially equivalent due to the polynomials being considered in applications where we essentially have very small preimages.}. To state our results we need to introduce the following function which counts the number of representations of an integer by a polynomial,
 \begin{equation}\label{numberrep}
   r_{P}(m,Q):= \#\{\mathbf{q} \sim Q;\ P(\mathbf{q})=m\},  \end{equation} 
 and the maximum over a dyadic box,
 \begin{equation}\label{maxnumberrep}
   r_{P}^*(Q):=  \max_{\mathbf{q} \sim Q}\left\{r_{P}(P(\mathbf{q}),Q)\right\}.
 \end{equation}  Several estimates on $\sum_{N,Q,P}$ already follow from the classic large sieve inequality that we recall now. A set of real numbers $\{x_k;\ k =1, \ldots,K\}$, is 
called {\it $\delta$-spaced modulo $1$\/} if $ \|x_k-x_j\|  \geq \delta$ for all $1 \le j < k\le K$,  where $\|x\|$ denotes the distance of a real number $x$ to its closest integer. Then by a result of Montgomery and Vaughan~\cite[Theorem~1]{MoVau}, we have 
\begin{equation}\label{eq:classic} 
\sum_{k=1}^K \left\vert \sum_{n=M+1}^{M+N} a_n \mathrm{e}(x_k n)\right\vert^2 \leq (\delta^{-1}+N)  \sum_{n=M+1}^{M+N} \vert a_n\vert^2.
\end{equation}

Similarly as observed by Zhao in \cite{Zhaoacta}, we remark that the inequality \eqref{eq:classic}
implies \eqref{largesieve} with
\begin{equation}\label{eq:trivialdelta} 
\Delta_{k,\ell}(Q,N) = \min\left\{r_{P}^*(Q)(Q^{2k}+N),Q^\ell(Q^k+N)\right\}.
\end{equation}  Due to the possible repetitions of the moduli $P(\mathbf{q})$, the factor $r_{P}^*(Q)$ appears in the inequality. This quantity is relatively harmless (say $r_{P}^*(Q)=Q^{o(1)}$) for the choice of polynomials appearing in  our applications (see Section\ \ref{adm} for details). 
Furthermore, heuristics from Zhao in \cite{Zhaoacta} make
clear that conjecturally, we shall have in the $\ell$-dimensional case
\begin{equation}
  \label{eq:Zhaoconj}
  \Delta_{k,\ell}(Q,N) = (QN)^{o(1)}r_{P}^*(Q)(Q^{\ell+k}+N),
\end{equation}
but we are still far away from proving this conjecture even for polynomials in one variable.
It was conjectured in \cite{AAKarin} that for the one-dimensional
case $\ell=1$, one should be able to reach
\[
  \Delta_{k,1}(Q,N)=(Q^{k+1}+Q^{1+1/(k-1)} N^{1-1/k(k-1)})(QN)^{o(1)}.
\]
This was proved by the second author as \cite[Thm.~1.2]{Munsch}
with  $k+1$ instead of $k-1$,
thus confirming \cite[Conjecture~21]{AAKarin} with $\omega=1/k(k+1)$.

Some improvements over the bound \eqref{eq:trivialdelta} have been obtained
by the first author using Weyl sums estimates and
the recent breakthroughs in Vinogradov's mean value theorem from \cite{BDG,Woo} and its generalizations 
to general polynomials of several variables \cite{PPW}.
More precisely it was proved in \cite[eq.~(5)]{Karinquart} that \eqref{largesieve} holds with
\begin{equation}
  \label{eq:Oldpolyls}
  \Delta_{k,\ell}(Q,N)=(Q^{\ell(k+1)} + Q^{\ell-1/2r_0\ell k}N
    + Q^{\ell+1/2r_0} N^{1-1/2r_0\ell k})(QN)^{o(1)},
  \end{equation}
where $r_0=\binom{\ell k+\ell-1}{\ell}-1$. 
Adapting the methods of \cite{BDG}, the recent impressive work \cite{GuoZh} extends \cite{PPW} and essentially
removes the factor $2$ from the bound in the multidimensional 
version of Vinogradov's mean value theorem. \footnote{The result is slightly more complicated to state and gives the expected number of solutions to multidimensional Vinogradov systems. The correcting factor for applications depends also on the number of variables $\ell$ but is very close to $2$.} By this result,
the factor $2$ in the exponent of \eqref{eq:Oldpolyls} 
could also be deleted, leading already to a slight improvement. 

In this note we go further and 
generalize the technique used in \cite{Munsch}, which leads
to a substantial improvement of the polynomial large sieve
inequality in several variables.
The best available bound of \cite{GuoZh}
in the multidimensional Vinogradov's mean value theorem
serves in the next section as an ingredient.

\subsection{Modular equations in boxes}\label{modular}
 
 To prove our main result, we need the following result
of Kerr \cite[Thm.~3.1]{Kerrboxes} which bounds the number of
solutions to multidimensional polynomial equations in boxes. This is a generalization of a result for polynomials in one variable \cite[Theorem $1$]{CGOS}.
We consider a polynomial of degree $k \geq 2$,
\[P(\mathbf{x}) =
  \sum_{0\leq \vert\mathbf{i}\vert \leq k} \alpha_\mathbf{i} \mathbf{x}^\mathbf{i},
  \ \alpha_\mathbf{i} \in \mathbb{Z}_{m},\] such that
\begin{equation}\label{defh} h_P:= \min_{\vert \mathbf{i}\vert =k}
  \vert \alpha_\mathbf{i} \vert=1, \end{equation}
where $\vert \mathbf{i}\vert$ is the sum of the components of
$\mathbf{i}\in\N_{0}^{\ell}$, and
$\mathbf{x}^{\mathbf{i}}=x_{1}^{i_{1}}\cdots x_{\ell}^{i_{\ell}}$.
Given any positive $K_1,\dots,K_\ell,L,H,R \geq 1$ and $(a,m)=1$,  we define by
$N(H,R;K_{1},\dots,K_{\ell},L)$ the number of solutions to the congruence
\begin{equation}\label{defsol}
 a P(\mathbf{x}) \equiv y \ (\md m),
\end{equation} where 
\[ (\mathbf{x},y) \in \prod_{i=1}^{\ell} [K_i+1,K_i+H] \times [L+1,L+R].\] 
\begin{lemma}\label{Kerrlem}
The following bound holds uniformly over arbitrary integers $K_{1},\dots,K_{\ell}$
and $L$:
\begin{equation}\label{Kerr} N(H,R;K_{1},\dots,K_{\ell},L) \ll H^\ell
\left((R/m)^{1/r(k+1)} +
(R/H^k)^{1/r(k+1)}\right)m^{o(1)}, \end{equation} where 
\begin{equation}
  \label{eq:rdef}
  r= {k+\ell \choose \ell}-1.
\end{equation}
\end{lemma}
Note that in \cite[Thm.~3.1]{Kerrboxes}, this result was stated with $2r$
instead of $r$ in \eqref{Kerr} since the result
\cite{PPW} was used. Indeed Kerr made use of \cite[Theorem $1.1$]{PPW} asserting that Vinogradov's systems of equations have at most the expected number of solutions as soon as the number of variables is large enough. Now thanks to the improvement of \cite{PPW} in \cite{GuoZh},
the number of necessary variables can be reduced by a factor $2$ \footnote{Again the correcting factor is more complicated and depends on the number of variables but is just slightly larger than $2$. For a more precise description of the gain here, see \cite[Theorem $3.2$]{PPW} and the enlightening discussion following equation $(1.4)$ in the same paper. For our purpose and to simplify the exposition we correct by a factor $2$.}. This impacts the improved bound \eqref{Kerr},
and thus our work, too.

\subsection{Large sieve inequality for multivariate
  polynomials}

  To begin with, we standardly define a subset of
Farey fractions
\[ \mathcal{S}(Q):=\left\{ \frac{a}{P(\mathbf{q})}; \ 
(a,P(\mathbf{q}))=1,\ 1\leq a < P(\mathbf{q}),\ \mathbf{q} \sim
Q\right\}.
\]
We would like to count the number of such fractions which are close to
each other. For this purpose we define
\begin{equation}\label{numberclose} M(N,Q)= \max_{x \in
\mathcal{S}(Q)}\# \left\{y \in \mathcal{S}(Q);\ \|x-y\| <1/2N
\right\}.\end{equation} 
We obtain the following bound depending only on the number of
variables and the degree of the polynomial $P$.

\begin{lemma}\label{estimateclose} Let $\ell$ be a positive integer and $P$ a polynomial of degree $k \geq 2$ with $h_P=1$ and such that $P(\uq)\gg Q^{k+o(1)}$ for all $\uq\sim Q$. For any integer
$N$ with $Q^{k} \leq N \leq Q^{2k}$, we have
\begin{equation}\label{estimatefrac} M(N,Q) \ll (QN)^{o(1)}
  Q^{\ell+k/r(k+1)}N^{-1/r(k+1)}, \end{equation}
with $r$ as in \eqref{eq:rdef}.
 \end{lemma}
 
 \begin{proof} Let $x=\frac{a}{P(\mathbf{q})}$ a fixed reduced fraction.
   We aim to estimate the number of $(\ell+1)$- tuple $(b,\mathbf{r})$
   with $(b,\mathbf{r})=1$ such that
   \begin{equation}\label{maj}\left\|\frac{a}{P(\mathbf{q})}
       -\frac{b}{P(\mathbf{r})}\right\|
     =\frac{\vert aP(\mathbf{r})-bP(\mathbf{q})\vert}%
     {P(\mathbf{q})P(\mathbf{r})} <1/2N.  \end{equation}
   Setting $z=aP(\mathbf{r})-bP(\mathbf{q})$, our problem is equivalent
   to estimating the number of $(\ell+1)$- tuples $(b,\mathbf{r})$ such
   that $\vert z\vert \ll Q^{2k}/N$. This boils down to counting the number
   of $(\ell+1)$- tuples $(\mathbf{r},z)$. The number of solutions is
   bounded above by the number of tuples with $\mathbf{r}\sim Q$
   and $\vert z\vert \ll Q^{2k}/N$ which are solutions to the congruence 
   \begin{equation}\label{cong}aP(\mathbf{r})
     = z \ (\,\bmod\, P(\mathbf{q})).\end{equation}
   Applying Lemma \ref{Kerrlem} to the polynomial
   $P$ with parameters $H=Q$, $R=Q^{2k}/N$,
   and $m=P(\mathbf{q})$ we deduce that the numbers of $(\mathbf{r},z)$
   satisfying \eqref{cong} is bounded above by 
   \[Q^{\ell+\varepsilon}(Q^k/N)^{1/r(k+1)} \]
   for any $\varepsilon>0$. This concludes the proof.
 
\end{proof} 
We follow a routine method to prove large sieve
inequalities and deduce the following theorem.
 \begin{theorem}\label{generalbound}
   Let $\ell$ be a positive integer and $P$ a polynomial of degree $k \geq 2$
  with $h_P=1$. For any integer
$N$ with $Q^{k} \leq N \leq Q^{2k}$ we have 
\begin{equation}\label{power2} \sum_{N,Q,P}^* \ll (QN)^{o(1)}
Q^{\ell+k/r(k+1)}N^{1-1/r(k+1)}\sum_{n=M+1}^{M+N} \vert
a_n\vert^2, \end{equation} with $r$ as in \eqref{eq:rdef} and where $\sum_{N,Q,P}^*$ denotes the sum $\sum_{N,Q,P}$ restricted  over those $\mathbf{q}$ such that $P(\mathbf{q}) \gg Q^{k+o(1)}.$
\end{theorem}

\begin{proof}
  We do not reproduce the full proof which is standard. This proof follows the
  method used for instance in the proofs of \cite[Theorem $1.2$]{Munsch}
  or \cite[Theorem $2$]{Zhaoacta} by incorporating the result of
  Lemma~\ref{estimateclose}. 
\end{proof}
 
\begin{remark} It is worth to note that general large sieve bounds have been obtained in \cite[Theorem $1.1$]{MS} in terms of the additive energy of the sequence of moduli. For specific choices of multivariate polynomials it might be possible to efficiently bound this additive energy and compare the resulting bounds with Theorem \ref{generalbound}. \end{remark}
 \begin{remark} For applications the estimate \eqref{power2} over ``good'' moduli $\uq$ is sufficient. Indeed for typical $\uq$ we have $P(\uq) \gg Q^{k+o(1)}$  or in other words the set of ``bad'' moduli has small density (see Lemma \ref{badmoduli} below). \end{remark}

  \section{A Bombieri--Vinogradov-type theorem with polynomial moduli}\label{BVsection}

  In this section, we prove a variant of the well-known
Bombieri--Vinogradov-theorem. Compared to the classical theorem, the version we look at deals
with moduli that are values of polynomials $P$
in several variables.
A first result in that direction has been established as
\cite[Thm.~1.2]{KarinBombieri}, where it is also discussed that
such a variant goes beyond the potential of the classical
Bombieri--Vinogradov-theorem, if the number of variables of $P$
is smaller than its degree. 
We obtain here an improvement of this result 
since it gives an extension of the moduli range and can handle a much
larger class of polynomials $P$ as described below. 
\begin{setting}\label{Psetting}
Let $P\in\Z[x_{1},\dots,x_{\ell}]$ be a 
polynomial of degree $k\geq 2$ with $h_P=1$
that is the product of $m$ irreducible
polynomials $H_{1},\dots,H_{m}\in\Z[x_{1},\dots,x_{\ell}]$,
each $H_{j}$ of degree $k_{j}\geq 1$  where $\ell_{j}$ denotes the number of variables in $H_{j}$. We assume that $h_{H_j}=1$ for all $1\leq j\leq m$ and also require that $r_{P}^*(Q) = Q^{o(1)}$.
Define as in \eqref{eq:rdef}
\begin{equation}
  \label{defrho}
  r=r_{k,\ell}:=\binom{k+\ell}{\ell}-1
  \text{ and } \rho=\rho_{k,\ell}:=r(k+1)/(r(k+1)-1).
\end{equation}
For $k_{j}$ and $\ell_{j}$ define the correspondent $r_{j}:=r_{k_{j},\ell_{j}}$
and $\rho_{j}:=\rho_{k_{j},\ell_{j}}$.
\end{setting}

We obtain the following theorem.
\begin{theorem}[A Bombieri--Vinogradov Theorem with
polynomial moduli]
\label{th:polyBV}  
Let $A,Q,x> 1$ be real,
and $P\in\Z[x_{1},\dots,x_{\ell}]$ a
polynomial of degree $k$ as in Setting~\ref{Psetting}.
Assume that the polynomials $H_{j}$ of degree $k_{j}$
have pairwise disjoint sets of 
\begin{equation}\label{eq:ellcond}
  \ell_{j}\geq \Big(1-\frac{1}{2\rho_{j}}\Big)k_{j}
  = \Big(1+\frac{1}{r_{j}(k_{j}+1)}\Big)\frac{k_{j}}{2}
\end{equation}
many variables for each $j=1,\dots,m$.
Further suppose that
\begin{equation}
\label{eq:QIV}
Q\leq x^{1/(2k+k/2\rho) -\eps}
\end{equation}
for arbitrary $\eps>0$. Let us recall the notation $$\psi(y;P(\uq),a):=\sum_{n\leq y \atop n\equiv a \bmod P(\uq)} \Lambda(n).$$
Then we have the estimate
\[
\sum_{\uq\sim Q} G_{\uq}\frac{\ph(P(\uq))}{Q^{\ell}}
   \sup_{y\leq x} \max_{\substack{a\md P(\uq)\\\gcd(a,P(\uq))=1}}
   |\psi(y;P(\uq),a)-y/\ph(P(\uq))|
   \ll_{A,\ell,k,m,\eps} \frac{x}{(\log x)^{A}},
\]
where 
\[
   G_{\uq}:= \mu^{2}(P(\uq)) \Lambda(H_{1}(\uq))\cdots
   \Lambda(H_{m}(\uq)),
\]
and where the sum runs
over all $\uq$ with $Q\leq q_{i}< 2Q$, $i=1,\dots,\ell$.
\end{theorem}

 \begin{remark} A condition on the number of variables like \eqref{eq:ellcond} was not appearing
in \cite{KarinBombieri} where the result was stated only for very special polynomials and an uniform large sieve bound (equivalently a single choice of $\Delta_{k,\ell}(Q,N)$) was used in the proof. This new condition appears in Lemma \ref{polyBMVT} below and comes essentially from the use of the standard large sieve bound \eqref{eq:trivialdelta} (which does not depend on $\ell$). It allows us to gain a substantial improvement in the level of distribution in comparison to \cite[Thm.~1.2]{KarinBombieri}, where the $Q$-exponent was around $3k$ whereas the exponent is here around $5k/2$. Note also that \cite[Thm.~1.2]{KarinBombieri} as well as
Theorem~\ref{th:polyBV} is only nontrivial
in the case when $\ell<k$, since otherwise, the
classical Bombieri--Vinogradov-theorem gives a
stronger statement. It is also not hard to see that our proof works also when the condition \eqref{eq:ellcond} is not fulfilled producing a slightly smaller level of distribution. \end{remark}

Since the proof of Theorem~\ref{th:polyBV} follows closely
Section~3 in \cite{KarinBombieri}, we restrict this
presentation to the changes that need to be made.
Briefly, the improvement comes from the sharpened polynomial large sieve
inequality obtained in Theorem~\ref{generalbound}, together with a
very careful case distinction. \\

Unlike in \cite{KarinBombieri} we assume only that the polynomial $P$ is the product of some irreducible
polynomials with some conditions on their variable sets and degrees,
that can be formulated in a much easier way.
The weight $G_{\uq}$ forces the factors $H_{i}(\uq)$
to attain prime values, so that all prime divisors of $P(\uq)$ are
values of polynomials, too. This property allows us to apply the fundamental Lemma \ref{polyBMVT} below to any divisor of the polynomial $P$.

\medskip

 In order to apply Theorem \ref{generalbound}, we need in the proof of Theorem \ref{th:polyBV} to discard the contribution of $\uq$ such that $P(\uq)$ is small. The following simple lemma allows us to control the number of such ``bad'' moduli.
\begin{lemma}\label{badmoduli}
Let $P \in
\mathbb{Z}[X_1,\dots,X_\ell]$ be a polynomial in $\ell$ variables of total degree $k \geq
2$. For every $\varepsilon>0$, we have
$$ \# \left\{ \uq \sim Q, |P(\uq)| \leq \varepsilon Q^k \right\} = O(\varepsilon^{1/k}Q^{\ell}),$$
where the implied constant only depends on $P$, $k$ and $\ell$.
\end{lemma}

\begin{proof}
We first prove the lemma when $P$ contains a summand $\lambda q_1^k$ with $\lambda \neq 0$. Indeed,  given any choice of $q_2,\dots,q_{\ell} \sim Q$, $P(\uq)$ is a nonzero integer polynomial of degree $k$ in $q_{1}$. We remark by factorizing this polynomial over $\mathbb{C}$ that the inequality $|P(\uq)| \leq \varepsilon Q^k$ forces $q_{1}$ to lie within $O\left(\varepsilon^{1/k} Q\right)$ of one of the (complex) roots of this polynomial. Hence, there is at most $ O(\varepsilon^{1/k}Q^{\ell})$ choices of such bad tuples $\uq$. The general case boils down to this situation after applying a linear change of variables to $P$ of the form $y_1=q_1,y_2=q_2+\lambda_2 q_1,\dots, y_{\ell}=q_{\ell}+\lambda_{\ell}q_1$ for a suitable choice of $(\lambda_2,\dots,\lambda_{\ell})$  depending only on the coefficients of $P$.
\end{proof} Setting $\varepsilon:=\varepsilon(Q)=(\log Q)^{-k(A+m+1)} $ and applying Lemma \ref{badmoduli}, we see that

 \[
\sum_{\uq\sim Q \atop |P(\uq)| \leq \varepsilon Q^k} G_{\uq}\frac{\ph(P(\uq))}{Q^{\ell}}
   \sup_{y\leq x} \max_{\substack{a\md P(\uq)\\\gcd(a,P(\uq))=1}}
   |\psi(y;P(\uq),a)-y/\ph(P(\uq))|
   \ll \frac{x}{(\log x)^{A}},
\] where we trivially bounded $G_{\uq}$ and $\psi(y;P(\uq),a)$. Hence, we can always assume that $P(\uq) \gg Q^{k+o(1)}$ in the rest of the proof and we will omit to precise it in the summations encountered. \\

The next Lemma is the key step in the proof and provides
an improved version of the polynomial version of the mean value
theorem \cite[Thm.~4.1]{KarinBombieri}
respectively \cite[Lemma~5.1]{KarinBombieri}.
\begin{lemma}[Polynomial Mean Value Theorem]
\label{polyBMVT}  
Let $Q,x>1$, 
and $P\in\Z[x_{1},\dots,x_{\ell}]$ be a polynomial of degree $k\geq 2$ with $h_P=1$, $r_P^*(Q)= Q^{o(1)}$
and such that the number of variables verifies $\ell\geq (1-1/2\rho)k$. 
For a character $\chi$ 
  we write $\displaystyle{\psi(x,\chi):=\sum_{n\leq x} \chi(n) \Lambda(n)}$.
  Then we have
\begin{equation}\label{sumcarac}
   \sum_{\uq\sim Q} \frac{P(\uq)}{\ph(P(\uq))}
     \sideset{}{^{*}}\sum_{\chi \bmod P(\uq)} \sup_{y\leq x}|\psi(y,\chi)|
   \ll (Qx)^{o(1)} Q^{\ell}x^{1-\delta'} 
 \end{equation} in the range $Q^{2k+k/2\rho+\varepsilon}\leq x$ for
a sufficiently small $\varepsilon>0$, with some $\delta'>0$
sufficiently small depending on $\varepsilon$. 
\end{lemma}

\begin{proof}
  Without loss of generality we can assume that $x$ lies in the interval
  $Q^{2k+\delta_{1}}\leq x\leq Q^{2k+\delta}$
  for some parameters $k/2\rho<\delta_{1}<\delta$.
  We can also assume that $\delta<k$, since  if $x\geq Q^{3k}$,
  the assertion follows already from \cite[Lemma~5.1]{KarinBombieri}.
  We follow the proof of \cite[Thm.~4.1]{KarinBombieri} and apply
  Vaughan's identity. To do so, let us introduce the parameter $U$ as
 \begin{equation}\label{eq:choiceU}
   U:=x/Q^{k+H} \end{equation}
 for an appropriate parameter $0<H<k$ to be chosen later.
 To obtain \eqref{sumcarac}, it suffices to find an upper bound for
  
\[S_1:=   \sum_{\uq\sim Q} \frac{P(\uq)}{\ph(P(\uq))} 
\sideset{}{^{*}}\sum_{\chi \bmod P(\uq)}\sum_{r\leq U }
\max_w \Big|\sum_{w<s\leq x/r}\chi(sr)  \Big| \]
  and 
  \[S_{M}:= \sum_{\uq\sim Q} \frac{P(\uq)}{\ph(P(\uq))} 
  \sideset{}{^{*}}\sum_{\chi \bmod P(\uq)} \Big| \sum_{m\sim M}
   \sum_{s\leq x/m} a(m)b(s)\chi(sm)\Big|,
 \] for any $M$ in the interval $U\leq M\leq W:=\max(U^{2},x/U)$. Here
 $a(n)$ and $b(n)$ are arithmetic functions depending on $U$ only
 and such that $ \vert b(n)\vert \leq  \tau(n)$,
 $\vert  a(n)\vert \leq \log n$ for all $n\in \mathbb{N}$. \\ 
 
 We bound $S_1$ using the P\'{o}lya--Vinogradov inequality and obtain
 \[
   S_1 \ll UQ^{\ell+3k/2+o(1)}=xQ^{\ell+k/2-H+o(1)}.
\]
   To ensure that $S_1 \ll Q^\ell x^{1-\delta'}$ for some $\delta'>0$, we need 
\begin{equation}\label{condH1}  k/2<H<k. \end{equation}
Bounding $S_M$ requires a more careful analysis and
the factor $\Delta_{k,\ell}(Q,x)$ appearing in the large sieve inequality
\eqref{largesieve} comes into play. Applying \cite[Lemma~3.2]{KarinBombieri},
which follows essentially from the inequality of Cauchy--Schwarz, we obtain
\begin{multline}
   \label{eq:eqI}
    S_{M}\ll Q^{o(1)}(\Delta_{k,\ell}(Q,M)M)^{1/2}
   (\Delta_{k,\ell}(Q,x/M)x/M)^{1/2}\\
   =Q^{o(1)}x^{1/2} \Delta_{k,\ell}(Q,M)^{1/2} \Delta_{k,\ell}(Q,x/M)^{1/2}.
 \end{multline}

We split the interval $[U,W]$ into
 three subintervals with boundaries at $Q^{k}$ and $x/Q^{k}$.
 In each of these subintervals, which we 
 call here regions, we estimate $S_{M}$ using different large sieve inequalities, or equivalently different values of $\Delta_{k,\ell}(Q,x)$ coming both from Theorem \ref{generalbound} and the trivial bound \eqref{eq:trivialdelta}. \\
 
 \textbf{First region:} $M\in [U,Q^{k}]$. \\

We apply the standard large sieve bound \eqref{eq:trivialdelta}
 \[\Delta_{k,\ell}(Q,M)\ll Q^{\ell}M+Q^{\ell+k}\ll Q^{\ell+k}.\]
 Since $Q^{k}\leq x/Q^{k}\leq x/M\leq x/U= Q^{k+H}\leq Q^{2k}$, 
 we can apply Theorem~\ref{generalbound} and take
 $\Delta_{k,\ell}(Q,x/M)=Q^{\ell+k/r(k+1)}(x/M)^{1-1/r(k+1)}$.
 Hence, it follows from \eqref{eq:eqI} that
 \begin{multline*}
   S_{M}\ll Q^{o(1)}x^{1/2} Q^{(\ell+k)/2}\cdot
   Q^{\ell/2+k/2r(k+1)} (x/M)^{1/2-1/2r(k+1)} \\
   \ll  Q^{o(1)}x^{1/2} Q^{\ell+k/2 +k/2r(k+1)
     +(k+H)(1/2-1/2r(k+1)) },
   \end{multline*}  
   where we used that $x/M\leq Q^{k+H}$ holds in this region.
  Inserting $x\leq Q^{2k+\delta}$, we have $S_M \ll Q^{\gamma+o(1)}$ where
   \[\gamma:= \ell+\frac{2k+\delta}{2}
      + \frac{k}{2}
     + \frac{k+H}{2}
     - \frac{H}{2r(k+1)}. \]
   To obtain \eqref{sumcarac} we need to show that
   $\gamma < \ell+2k+\delta_1$.
   Choosing $\delta$ such that $\delta-\delta_1$ is sufficiently small,
   it suffices to show that $\gamma < \ell+2k+\delta$.
   This holds true as soon as 
   \begin{equation*}
 2k +\frac{\delta}{2}  +\frac{H}{2}  - \frac{H}{2r(k+1)} < 2k+\delta,
\end{equation*} 
i.e.\ if
\begin{equation}
  \label{eq:eH}
  H \Big(1-\frac{1}{r(k+1)}\Big)= H/\rho <  \delta.
\end{equation}

Under this additional restriction on $H$, this implies
 $S_{M} \ll Q^{\ell}x^{1-\delta'}$ for any sufficiently small $\delta'>0$. \\
   
   \textbf{Second region}: $M\in[Q^{k},x/Q^{k}]$ .\\
   
   In this region we have $Q^{k}\leq x/M \leq x/Q^{k}\leq Q^{2k}$,
   since we assumed $x\leq Q^{3k}$. Then Theorem~\ref{generalbound}
   applies to both the sum over the intervals of size $M$ and $x/M$.
   From \eqref{eq:eqI}, this yields the bound
 \begin{multline*}
   S_{M}\ll Q^{o(1)}x^{1/2} \Delta_{k,\ell}(Q,M)^{1/2}\Delta_{k,\ell}(Q,x/M)^{1/2}\\
   \ll Q^{o(1)}x^{1/2}Q^{\ell+k/r(k+1)}M^{1/2-1/2r(k+1)} (x/M)^{1/2-1/2r(k+1)}\\
   =Q^{o(1)} x^{1-1/2r(k+1)} Q^{\ell+k/r(k+1)}
   = Q^{o(1)}xQ^{\ell} (Q^{2k}/x)^{1/2r(k+1)},
 \end{multline*}
 and having $Q^{2k+\delta_{1}}\leq x$, the last term is
 $\ll Q^{\ell}x^{1-\delta'}$ for any sufficiently small $\delta'>0$,
 as was to be shown. \\
   
   \textbf{Third region:}  $M\in[x/Q^{k},W]$ with $W=\max\{U^{2},x/U\} $. \\

 $\bullet$  Consider first the case $W=x/U$, so that
 $M\leq W\leq Q^{2k}$ by the choice of $U$. Theorem~\ref{generalbound}
 can be applied for the sum over $m\sim M$, and we use the standard bound
 $\Delta_{k,\ell}(Q,x/M)\ll Q^{\ell+k}$ following from \eqref{eq:trivialdelta} and $x/M\ll Q^{k}$.
 Remark that this situation is symmetric to the one in the first region.
 Therefore we use \eqref{eq:eqI} from above which yields the bound
\begin{multline*}
  S_{M}\ll Q^{o(1)}x^{1/2}\Delta_{k,\ell}(Q,M)^{1/2}
  \Delta_{k,\ell}(Q,x/M)^{1/2}\\
  \ll Q^{o(1)} x^{1/2}Q^{(\ell+k)/2}\cdot Q^{\ell/2+k/2r(k+1)}
  M^{1/2-1/2r(k+1)}.
\end{multline*}

Now inserting $M\leq W=x/U= Q^{k+H}$, we
get the same upper bound for $S_{M}$
as in the first region, so we are done in this case by symmetry.

$\bullet$   If otherwise $W=U^{2}$,
then we  use again the standard bound $\Delta_{k,\ell}(Q,x/M)\ll Q^{\ell+k}$ from \eqref{eq:trivialdelta}.
Also from \eqref{eq:trivialdelta}, we know the bound
\[\Delta_{k,\ell}(Q,M)\ll r_P^*(Q)(M+Q^{2k}) \ll Q^{o(1)}(M+Q^{2k}).\]

We split further the discussion into two subcases depending on
whether $M \leq Q^{2k}$ or not. In the former case we have 
\[S_M \ll Q^{o(1)}x^{1/2}
  \Delta_{k,\ell}(Q,M)^{1/2}\Delta_{k,\ell}(Q,x/M)^{1/2}
  \ll Q^{o(1)}x^{1/2}Q^{k} Q^{(\ell+k)/2}.
\]
Inserting $x\leq Q^{2k+\delta}$, we have $S_M \ll Q^{\gamma_1+o(1)}$ where
\[\gamma_1:= \frac{1}{2}(3k+\delta-\ell)+\ell+k.\]
 To infer \eqref{sumcarac}, we argue as in the first region and thus we
need to show that  $\gamma_1<\ell+2k+\delta$.
A quick computation reveals that this is equivalent to 
\[
  k-\ell <  \delta. 
\]
This holds true under the hypothesis $\ell\geq (1-1/2\rho)k$
and $\delta>k/2\rho$. \\

In the latter case $M \geq Q^{2k}$, recall that
$M\ll U^{2}=x^{2}Q^{-2k-2H}$. As before,
using \eqref{eq:eqI}, we arrive at

\[
  S_{M}\ll Q^{o(1)} x^{1/2}Q^{(\ell+k)/2}M^{1/2}\ll Q^{o(1)} x^{1/2}
  Q^{(\ell+k)/2} xQ^{-k-H}.
\]
Inserting $x\leq Q^{2k+\delta}$, we have $S_M \ll Q^{\gamma_2+o(1)}$ where
\[
  \gamma_2:= 2k+\delta+\ell/2+k/2 + \delta/2-H.
\]
Arguing as above, to deduce \eqref{sumcarac},
we need to show that  $\gamma_2<\ell+2k+\delta$. This holds if
\[
   k/2+\delta/2-H < \ell/2,
\]
i.e.\
\begin{equation}
  \label{eq:Hb}
  H> \delta/2+(k-\ell)/2.
\end{equation} Under this additional restriction on $H$, this implies
 $S_{M} \ll Q^{\ell}x^{1-\delta'}$ for any sufficiently small $\delta'>0$. \\
   
   To finish the proof we need to choose the parameter $H$ subject to
   the restrictions \eqref{condH1}, \eqref{eq:eH} and \eqref{eq:Hb}.
   The hypothesis $\delta_1 >k/2\rho$ ensures that both \eqref{condH1}
   and \eqref{eq:eH} can hold together. Now under the condition
   $\ell\geq (1-1/2\rho)k$ we have $\delta/2+(k-\ell)/2 < \delta$
   so \eqref{eq:Hb} holds too. This concludes the proof.

\end{proof}

\subsection{Conjectural level of distribution}
The question arises of the level of distribution that can be expected under the assumption
of the multidimensional analogue \eqref{eq:Zhaoconj} of Zhao's conjecture. We give some hints
about the changes to be made under this assumption. The application of Theorem~\ref{generalbound} in the proof of Lemma~\ref{polyBMVT}  
is then replaced by the application of \eqref{eq:Zhaoconj}.
Hence in the first region, we get 
$\Delta_{k,\ell}(Q,x/M)\ll (Q^{\ell+k}+Q^{k+H})Q^{o(1)}\ll Q^{k+\max\{\ell,H\}}Q^{o(1)}$
and
$\gamma=k+\delta/2+(k+\ell)/2+k/2+\max\{\ell,H\}/2$, which is
$<\ell+2k+\delta$ if we choose
\begin{equation}
  \label{eq:Hzhao}
  H<\ell+\delta.
\end{equation}
This new condition on $H$ replaces condition \eqref{eq:eH}.
In the third region with $W=U^2$, the bound $\Delta_{k,\ell}(Q,M) \ll (Q^{\ell+k}+M)Q^{o(1)}$ leads to the new subcase distinction $M\leq Q^{\ell+k}$ and $M>Q^{\ell+k}$. In the former case,
$\gamma_{1}=\ell+2k+\delta/2$ is already admissible, and the treatment
of the second case does not change.  To guarantee that $H$ can be chosen subject to
the restrictions \eqref{condH1}, \eqref{eq:Hb} and \eqref{eq:Hzhao} 
we need that $\delta+\ell>k/2$ as well as
$\delta+\ell>\delta/2+(k-\ell)/2$, leading to
$\delta>\max\{k/2-\ell,k-3\ell\}$. Therefore when $\ell \geq k/2$, we can choose any $\delta>0$ and \eqref{eq:Zhaoconj} leads to an analogue of Lemma~\ref{polyBMVT} in the range
\begin{equation}
  \label{eq:ZhaoQbound}
Q^{2k+\eps}\leq x
\end{equation} reaching the optimal expected range. It is worth to notice that the method for $\ell=1$ seems to lead to a Bombieri--Vinogradov theorem in the range $Q^{2k+\eps} \leq x$ only when $k=2$ (it was also proved unconditionally by Baker \cite{Baker2}), while heuristically this range should be reached for any degree.

\bigskip
\subsection{Proof of Theorem~\ref{th:polyBV}}
  Due to the assigned weight $G_{\uq}$, each $P(\uq)$
  on the left hand side of the assertion is squarefree.
  The tuple $\uq$ has therefore the property that each
  $H_{j}(\uq)$ equals to a prime, and these
  primes $H_{1}(\uq),\dots,H_{m}(\uq)$ are pairwise different.

  Therefore, a divisor $d$ of $P(\uq)$ must be
  $d=1$ or $d=\tilde{P}(\uq)$
  for some polynomial $\tilde{P}$ that divides $P$ in $\Z[\ux]$.

  Following the proof of \cite[Thm.~1.2]{KarinBombieri},
  our task is reduced to give a proof of the bound
  \begin{equation}
    \label{eq:Gqsum}
     \sum_{\substack{\tilde{P}\mid P\\\tilde{P}\neq 1}}\sum_{\uq\sim Q} G_{\uq} Q^{-\ell}
     \sideset{}{^{*}}\sum_{\chi \bmod \tilde{P}(\uq)}
     \sup_{y\leq x}|\psi(y,\chi)| \ll x^{1-\delta}
   \end{equation}
   for any small $\delta>0$,
   where $\chi$ runs through the primitive characters
   mod $\tilde{P}(\uq)$, denoted by the star at the sum.
   
   The key observation is that  Lemma~\ref{polyBMVT} can be applied 
   to each nonconstant $\tilde{P}\mid P$ 
  (remark that for divisors of degree $1$, it follows directly from the classical Bombieri--Vinogradov theorem). Indeed, we need to verify the hypotheses of Lemma~\ref{polyBMVT}. First of all, a short calculation shows that
   \begin{equation}
     \label{eq:Qexpon}
     1/(2k+k/2\rho)\leq 1/(2\tilde{k}+\tilde{k}/2\tilde{\rho})
   \end{equation}
   holds true for all nonconstant $\tilde{P}\mid P$ with corresponding
     degree $\tilde{k}\leq k$. Hence the variable $Q$ belongs to the admissible range. Denote by $\tilde{\ell}$ the number of variables of $\tilde{P}$. Our assumption \eqref{eq:ellcond} and the fact that we know that the $H_{j}$ have disjoint sets of variables implies by additivity
     \[\tilde{\ell}\geq \Big(1-\frac{1}{2\tilde{\rho}}\Big)\tilde{k}
       = \Big(1+\frac{1}{\tilde{r}(\tilde{k}+1)}\Big)\frac{\tilde{k}}{2}
     \] for any nonconstant factor $\tilde{P}\mid P$.
        
   Hence, the application of Lemma~\ref{polyBMVT} is possible and we can bound
    the left hand side of \eqref{eq:Gqsum} by
   \[ \sum_{\tilde{P}\mid P} x^{1-\delta'}(\log x)^{m}\ll_{m} x^{1-\delta}
   \]
   for any small $\delta>0$, 
   where we used that there are only $\ll_{m}1$ many divisor polynomials
   of $P$ in $\Z[\ux]$.  The theorem follows.

\subsection{Choice of admissible polynomials}\label{adm}
In order to be meaningful, our result needs to be applied to polynomials $H_j$ taking a lot of prime values. In one variable, the Bunyakovsky conjecture states that, under mild conditions, any irreducible polynomial should take infinitely many prime values. Its quantitative version, the  Bateman--Horn conjecture, predicts the frequency of prime values in polynomial sequences. Even though such a statement remains widely open for any fixed polynomial of degree greater than $2$, it has been recently proved that Bateman-Horn conjecture holds for $100\%$ of polynomials \cite{SoSko}. For polynomials of several variables, the situation is different and only a few examples of polynomials taking infinitely many prime values are known. For instance, this has been proved using sieve methods by Friedlander and Iwaniec \cite{FIW} for $x^2+y^4$ or by Heath-Brown \cite{HB} for $x^3+2y^3$.  \\

The reader might wonder if any such good choice is possible for our problem. Indeed, we can choose the polynomials $H_j$ explicitly as norm form polynomials,  i.e.\ $H_{j}(\ux_{j})=N_{K}(\ux_{j})$,
for a certain number field $K$ over $\mathbb{Q}$ and a set of
$\ell_{j}$ many variables. Inspired by the aforementioned works, Maynard proved in \cite{Maynard} that these polynomials takes the expected number of prime values when $\ell_{j}\geq 3k_{j}/4$. Precisely, the number
of $\uq_{j}\sim Q$ such that $H_{j}(\uq_{j})=N_{K}(\uq_{j})$ is prime is
$\gg Q^{\ell}/\log Q$ by \cite[Thm.~1.1]{Maynard}.
(Note that our condition
$\ell_{j}\geq (1-1/2\rho_{j})k_{j}$ holds true assuming $\ell_{j}>3k_{j}/4$.) 

Thus, this result shows that for this choice of polynomials the number of moduli that are admissible in Theorem~\ref{th:polyBV} is
$\gg Q^{\ell}/\log Q$. This prevents Theorem~\ref{th:polyBV} to follow trivially from an application of the classical Bombieri--Vinogradov-Theorem. Furthermore these polynomials verify the hypothesis $r_{N_K}^*(Q) \ll Q^{o(1)}$ of Setting~\ref{Psetting} on the number of possible repetitions. Indeed, the specific structure of the polynomials $N_K$ written as a norm allows to apply the divisor bound in number fields \footnote{See also the discussion on mathoverflow: https://mathoverflow.net/questions/68437/the-divisor-bound-in-number-fields} \cite[Proposition $2.5$]{Chang} or more precise results by Schmidt \cite[Theorem $3$]{Schmidt} on norm form equations. As a consequence, let us mention the following result: 

\begin{cor}\label{Maynardcor} Let $n, k$ be positive integers. Let $f \in  \mathbb{Z}[X]$ be a monic irreducible
polynomial of degree $n$ with root $\omega \in \mathbb{C}$. Let $K = \mathbb{Q}(\omega)$ be the corresponding number field of degree $n$, and let $N_{K} \in  \mathbb{Z}[X_1,\dots,X_{n-k}]$ be the incomplete norm form
\[
  N_K(q_1,\dots,q_{n-k})=N_{K/\mathbb{Q}}\left( \sum_{i=1}^{n-k} q_i \omega^{i-1} \right).
\]
Then if $n \geq 4k$, there exist infinitely many primes $p$ such that $p-1$ has a prime divisor $d=N_{K}(q_1,\dots,q_{n-k})$ of size $\gg p^{2/5+o(1)}.$

\end{cor}
\begin{proof}
  The method to prove such a result from a Bombieri--Vinogradov theorem is classical and can be found for instance in \cite[Thm.~5]{BZsparse}. Hence, we will not give all the details. Mainly speaking,  it is sufficient to consider the following sum
\[
  \sum_{x<n\leq 2x} \Lambda(n+1) \sum_{\uq  \sim Q \atop N_K(\uq) \mid n} \Lambda(N_K(\uq))
\]
  and switch the summation. Then \cite[Thm.~1]{Maynard} asserts that the polynomial $N_K$ takes the expected number of prime values and we can conclude using Theorem \ref{th:polyBV} in the admissible range $Q \leq x^{2/5n+o(1)}.$
\end{proof}

 A similar result could have been stated for primes $p$ such that $p-1$ has a prime divisor $d=x^3+2y^3 \gg p^{72/179+o(1)}$ using Heath-Brown's result combined with Theorem \ref{th:polyBV}. Here the exponent $72/179$ comes from the explicit computation of $1/(2k+k/2\rho)$  in Theorem \ref{th:polyBV} and gives a slight improvement over $2/5$.


\section*{Acknowledgements}
The second author would like to thank the Max Planck Institute for
Mathematics, Bonn, and the University of D\"{u}sseldorf  for support and hospitality during his work on this project.  The second author also acknowledges support of the Austrian Science Fund (FWF), stand-alone project P 33043  ``Character sums, L-functions and applications''.

\section*{Data availability statement}

Data sharing not applicable to this article as no datasets were generated or analysed during the current study.


\begin{thebibliography}{10}

\bibitem{BZsparse}
S.~Baier and L.~Zhao.
\newblock Bombieri-{V}inogradov type theorems for sparse sets of moduli.
\newblock {\em Acta Arith.}, 125(2):187--201, 2006.

\bibitem{Baker2}
R.~Baker.
\newblock Primes in arithmetic progressions to spaced moduli. {III}.
\newblock {\em Acta Arith.}, 179(2):125--132, 2017.

\bibitem{BakerPP1}
R.~C. Baker.
\newblock Primes in arithmetic progressions to spaced moduli.
\newblock {\em Acta Arith.}, 153(2):133--159, 2012.

\bibitem{MS}
R.~C. Baker, M.~Munsch and I.~E. Shparlinski.
\newblock Additive energy and a large sieve inequality for sparse sequences.
\newblock {\em Preprint}, https://arxiv.org/abs/2103.12659.

\bibitem{BPS}
W.~D. Banks, F.~Pappalardi, and I.~E. Shparlinski.
\newblock On group structures realized by elliptic curves over arbitrary finite
  fields.
\newblock {\em Exp. Math.}, 21(1):11--25, 2012.

\bibitem{BDG}
J.~Bourgain, C.~Demeter, and L.~Guth.
\newblock Proof of the main conjecture in {V}inogradov's mean value theorem for
  degrees higher than three.
\newblock {\em Ann. of Math. (2)}, 184(2):633--682, 2016.

\bibitem{FermatBFKS}
J.~Bourgain, K.~Ford, S.~V. Konyagin, and I.~E. Shparlinski.
\newblock On the divisibility of {F}ermat quotients.
\newblock {\em Michigan Math. J.}, 59(2):313--328, 2010.

\bibitem{Chang}
M.~Chang.
\newblock Factorization in generalized arithmetic progressions and application
  to the {E}rd{\H{o}}s-{S}zemer{\'e}di sum-product problems.
\newblock {\em Geometric Functional Analysis GAFA}, 13(4):720--736, 2003.

\bibitem{CGOS}
J.~Cilleruelo, M.~Z. Garaev, A.~Ostafe, and I.~E. Shparlinski.
\newblock On the concentration of points of polynomial maps and applications.
\newblock {\em Math. Z.}, 272(3-4):825--837, 2012.

\bibitem{FIW}
J.~Friedlander and H.~Iwaniec.
\newblock The polynomial {$X^2+Y^4$} captures its primes.
\newblock {\em Ann. of Math. (2)}, 148(3):945--1040, 1998.

\bibitem{Gall}
P.~X. Gallagher.
\newblock The large sieve and probabilistic {G}alois theory.
\newblock In {\em Analytic number theory ({P}roc. {S}ympos. {P}ure {M}ath.,
  {V}ol. {XXIV}, {S}t. {L}ouis {U}niv., {S}t. {L}ouis, {M}o., 1972)}, pages
  91--101, 1973.

\bibitem{GuoZh}
S.~Guo and R.~Zhang.
\newblock On integer solutions of {P}arsell-{V}inogradov systems.
\newblock {\em Invent. Math.}, 218(1):1--81, 2019.

\bibitem{Karinquart}
K.~Halupczok.
\newblock Large sieve inequalities with general polynomial moduli.
\newblock {\em Q. J. Math.}, 66(2):529--545, 2015.

\bibitem{KarinBombieri}
K.~Halupczok.
\newblock A {B}ombieri-{V}inogradov theorem with products of {G}aussian primes
  as moduli.
\newblock {\em Funct. Approx. Comment. Math.}, 57(1):77--91, 2017.

\bibitem{AAKarin}
K.~Halupczok.
\newblock Bounds for discrete moments of {W}eyl sums and applications.
\newblock {\em Acta Arith.}, 194(1):1--28, 2020.

\bibitem{HB}
D.~R. Heath-Brown.
\newblock Primes represented by {$x^3+2y^3$}.
\newblock {\em Acta Math.}, 186(1):1--84, 2001.

\bibitem{Kerrboxes}
B.~Kerr.
\newblock Solutions to polynomial congruences in well-shaped sets.
\newblock {\em Bull. Aust. Math. Soc.}, 88(3):435--447, 2013.

\bibitem{Matomaki}
K.~Matom\"{a}ki.
\newblock A note on primes of the form {$p=aq^2+1$}.
\newblock {\em Acta Arith.}, 137(2):133--137, 2009.

\bibitem{Maynard}
J.~Maynard.
\newblock Primes represented by incomplete norm forms.
\newblock {\em Forum Math. Pi}, 8:e3, 128, 2020.

\bibitem{Mer}
J.~Merikoski.
\newblock On the largest square divisor of shifted primes.
\newblock {\em Acta Arith.}, 196(4):349--386, 2020.

\bibitem{MoVau}
H.~L. Montgomery and R.~C. Vaughan.
\newblock The large sieve.
\newblock {\em Mathematika}, 20(2):119--14, 1973.

\bibitem{Munsch}
M.~Munsch.
\newblock A large sieve inequality for power moduli.
\newblock {\em Acta Arith.}, 197(2):207--211, 2021.

\bibitem{PPW}
S.~T. Parsell, S.~M. Prendiville, and T.~D. Wooley.
\newblock Near-optimal mean value estimates for multidimensional {W}eyl sums.
\newblock {\em Geom. Funct. Anal.}, 23(6):1962--2024, 2013.

\bibitem{Schmidt}
W.~M. Schmidt.
\newblock The number of solutions of norm form equations.
\newblock {\em Transactions of the American Mathematical Society},
  317(1):197--227, 1990.

\bibitem{IgorFermat}
I.~E. Shparlinski.
\newblock Fermat quotients: exponential sums, value set and primitive roots.
\newblock {\em Bull. Lond. Math. Soc.}, 43(6):1228--1238, 2011.

\bibitem{SZell}
I.~E. Shparlinski and L.~Zhao.
\newblock Elliptic curves in isogeny classes.
\newblock {\em J. Number Theory}, 191:194--212, 2018.

\bibitem{SoSko}
A.~N. Skorobogatov and E.~Sofos.
\newblock Schinzel hypothesis with probability 1 and rational points.
\newblock {\em Preprint}, https://arxiv.org/abs/2005.02998.

\bibitem{Woo}
T.~D. Wooley.
\newblock Vinogradov's mean value theorem via efficient congruencing.
\newblock {\em Ann. of Math. (2)}, 175(3):1575--1627, 2012.

\bibitem{Zhaoacta}
L.~Zhao.
\newblock Large sieve inequality with characters to square moduli.
\newblock {\em Acta Arith.}, 112(3):297--308, 2004.

\end{thebibliography}
\end{document}